\documentclass[11pt,oneside,english,final]{amsart}

\usepackage[notcite,notref,color]{showkeys}  

\usepackage{amstext,amsthm,amssymb,mathabx,mathtools}
\usepackage[T1]{fontenc}
\usepackage[latin1]{inputenc}
\usepackage{geometry}
\usepackage{graphicx}

\usepackage[colorlinks,bookmarks,linkcolor=black,citecolor=black]{hyperref} 
\usepackage{enumitem}
\usepackage[english,algoruled,lined,noresetcount,norelsize]{algorithm2e}

\newtheorem{theorem}{Theorem}
\newtheorem{lemma}[theorem]{Lemma}

\theoremstyle{remark}

\newcommand{\oldqed}{}
\newcommand{\qedClaim}{\hfill\scalebox{.6}{$\Box$}}

\setlist{itemsep=2pt,parsep=1pt,topsep=3pt,partopsep=0pt}  
\setenumerate{leftmargin=*,labelindent=\parindent} 

\DeclarePairedDelimiter{\size}{|}{|}

\renewcommand{\rho}{\varrho}
\renewcommand{\subset}{\subseteq}

\usepackage{tikz}
\usetikzlibrary{math}
\usetikzlibrary{calc,decorations.pathmorphing,through}
\tikzset{snake it/.style={decorate, decoration=snake}}

\definecolor{DarkDesaturatedBlue}{HTML}{3A3556}
\definecolor{VividOrange}{HTML}{F15918}
\definecolor{PureOrange}{HTML}{FFBA00}
\definecolor{LightGrayishPink}{HTML}{EEC5D5}
\definecolor{VerySoftBlue}{HTML}{B5AFDB}

\newcommand{\triple}[7]{
 
 \ifx\relax#4\relax
 \def\qoffs{0pt}
 \else
 \def\qoffs{#4}
 \fi
 
 \def\qhedge{
  ($#1+#3!\qoffs!-90:#2-#3$) --
  ($#2+#1!\qoffs!-90:#3-#1$) --
  ($#3+#2!\qoffs!-90:#1-#2$) -- cycle}
 
 \coordinate (12) at ($#1!\qoffs!90:#2$);
 \coordinate (13) at ($#1!\qoffs!-90:#3$);
 \coordinate (23) at ($#2!\qoffs!90:#3$);
 \coordinate (21) at ($#2!\qoffs!-90:#1$);
 \coordinate (31) at ($#3!\qoffs!90:#1$);
 \coordinate (32) at ($#3!\qoffs!-90:#2$);
 
 \def\nqhedge{
  (13) let \p1=($(13)-#1$), \p2=($(12)-#1$) in
  arc[start angle={atan2(\y1,\x1)}, delta angle={atan2(\y2,\x2)-atan2(\y1,\x1)-360*(atan2(\y2,\x2)-atan2(\y1,\x1)>0)}, x radius=\qoffs, y radius=\qoffs] --
  (21) let \p1=($(21)-#2$), \p2=($(23)-#2$) in
  arc[start angle={atan2(\y1,\x1)}, delta angle={atan2(\y2,\x2)-atan2(\y1,\x1)-360*(atan2(\y2,\x2)-atan2(\y1,\x1)>0)}, x radius=\qoffs, y radius=\qoffs] --
  (32) let \p1=($(32)-#3$), \p2=($(31)-#3$) in
  arc[start angle={atan2(\y1,\x1)}, delta angle={atan2(\y2,\x2)-atan2(\y1,\x1)-360*(atan2(\y2,\x2)-atan2(\y1,\x1)>0)}, x radius=\qoffs, y radius=\qoffs] --
  cycle}
 
 \ifx\relax#5\relax
 \def\qlwidth{1pt}
 \else
 \def\qlwidth{#5}
 \fi
 
 \ifx\relax#7\relax
 \fill \nqhedge;
 \else
 \fill[#7]\nqhedge;
 \fi
 
 \ifx\relax#6\relax
 \draw[line width=\qlwidth,rounded corners=\qoffs]\nqhedge;
 \else
 \draw[line width=\qlwidth,#6]\nqhedge;
 \fi
}

\title[Ramsey numbers for $1$-degenerate $3$-graphs]{Ramsey numbers for $1$-degenerate $3$-graphs}

\author[P. Allen]{Peter Allen}
\address{(PA) London School of Economics, Department of Mathematics, Houghton Street, London WC2A 2AE, UK}
\email{p.d.allen@lse.ac.uk}
\author[S. Boyadzhiyska]{Simona Boyadzhiyska}
\address{(SB) HUN-REN Alfr\'ed R\'enyi Institute of Mathematics, Budapest, Hungary}
\email{simona@renyi.hu}

\author[M. Pavez-Sign\'e]{Mat\'ias Pavez-Sign\'e}
\address{(MPS) Mathematical Engineering Department, University of Chile, and Center for Mathematical Modeling (CNRS IRL2807), Chile.}
\email{mpavez@dim.uchile.cl}

\date{}

\begin{document}

\begin{abstract}
 We construct a $3$-uniform $1$-degenerate hypergraph on $n$ vertices whose $2$-colour Ramsey number is $\Omega\big(n^{3/2}/\log n\big)$. This shows that all remaining open cases of the hypergraph Burr-Erd\H{o}s conjecture are false. Our graph is a variant of the celebrated hedgehog graph. We additionally show near-sharp upper bounds, proving that all $3$-uniform generalised hedgehogs have $2$-colour Ramsey number $O\big(n^{3/2}\big)$.
\end{abstract}

\maketitle

\section{Introduction}

Given two $r$-uniform graphs ($r$-graphs) $G$ and $H$, the \emph{(2-colour) Ramsey number} of $G$ and~$H$, denoted $R(G,H)$, is the smallest number $n\in\mathbb N$ so that any blue/red-colouring of the edges the complete $r$-graph on $n$ vertices contains a blue copy of $G$ or a red copy of $H$. If~$G=H$, then we simply write $R(G)=R(G,H)$ and say that $R(G)$ is the (2-colour) Ramsey number of $G$. Ramsey numbers for more colours are defined analogously: we seek a monochromatic copy of an {$r$-graph} $G$ in a $k$-edge-coloured complete $r$-graph and refer to the corresponding parameter as the \emph{$k$-colour Ramsey number}.

As usual, the \emph{degree} of a vertex in a hypergraph is the number of edges it is contained in.
An $r$-graph $G$ is said to be \emph{$D$-degenerate} if any subhypergraph $H$ of $G$ has a vertex of degree at most $D$ in $H$. An influential conjecture of Burr and Erd\H{o}s~\cite{BurErd} from the 70s, proved by Lee~\cite{Lee} in 2017, is that for any constants $D$ and $k$ there is a constant $C$ such that any $n$-vertex $D$-degenerate $2$-graph has $k$-colour Ramsey number at most $Cn$. 

A natural question, the \emph{hypergraph Burr-Erd\H{o}s conjecture}, is then whether a similar statement holds for uniformity $r\ge 3$. This is false: the first disproof was due to Kostochka and R\"odl~\cite{KosRod} who found a $4$-uniform $1$-degenerate counterexample in the 2-colour setting; it follows that the statement also fails for higher uniformities, degeneracies and number of colours. Uniformity three is less understood: Kostochka and R\"odl found counterexamples in the 2-colour setting for sufficiently large degeneracy, and Conlon, Fox and R\"odl~\cite{CFR} found a $1$-degenerate counterexample for three colours. Recently, Dubroff, Gir\~ao, Hurley and Yap~\cite{DGHY} found a counterexample of degeneracy only eight in the 2-colour case. This leaves open (only) the situation for $3$-uniform hypergraphs with degeneracy at most seven in two colours.

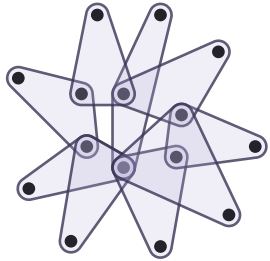
\begin{figure}[h!]
    \centering\begin{tikzpicture}[thick, scale=.7]
    \tikzstyle{every node}=[circle, draw, fill, inner sep=0pt, minimum width=4pt]
    
    \draw node (b1) at (0.3,0.2) {};
    \draw node (b2) at (1,-0.2) {};
    \draw node (b3) at (2,0) {};
    \draw node (b4) at (2.1,0.8) {};
    \draw node (b5) at (1,1.2) {};
    \draw node (b6) at (0.2,1.2) {};

    
    \draw node (s8) at (1.7,2.7) {};
    \triple{(b5)}{(s8)}{(b2)}{6pt}{1pt}{DarkDesaturatedBlue,opacity=0.8}{VerySoftBlue,opacity=0.2};

    \draw node (s1) at (-1,1.5) {};
    \triple{(s1)}{(b6)}{(b1)}{6pt}{1pt}{DarkDesaturatedBlue,opacity=0.8}{VerySoftBlue,opacity=0.2};
    
    \draw node (s2) at (-.8,-.6) {};
    \triple{(b1)}{(b2)}{(s2)}{6pt}{1pt}{DarkDesaturatedBlue,opacity=0.8}{VerySoftBlue,opacity=0.2}; 
    
    \draw node (s3) at (0,-1.6) {};
    \triple{(b1)}{(b2)}{(s3)}{6pt}{1pt}{DarkDesaturatedBlue,opacity=0.8}{VerySoftBlue,opacity=0.2};

    \draw node (s4) at (1.7,-1.7) {};  
    \triple{(s4)}{(b2)}{(b3)}{6pt}{1pt}{DarkDesaturatedBlue,opacity=0.8}{VerySoftBlue,opacity=0.2}; 

    \draw node (s5) at (3.5,0.2) {};
    \triple{(b4)}{(s5)}{(b3)}{6pt}{1pt}{DarkDesaturatedBlue,opacity=0.8}{VerySoftBlue,opacity=0.2};

    \draw node (s6) at (2.8,2) {};
    \triple{(s6)}{(b4)}{(b5)}{6pt}{1pt}{DarkDesaturatedBlue,opacity=0.8}{VerySoftBlue,opacity=0.2};

    \draw node (s7) at (.5,2.7) {};
    \triple{(b5)}{(b6)}{(s7)}{6pt}{1pt}{DarkDesaturatedBlue,opacity=0.8}{VerySoftBlue,opacity=0.2};
    
    \draw node (s8) at (3,-1.1) {};
    \triple{(b2)}{(b4)}{(s8)}{6pt}{1pt}{DarkDesaturatedBlue,opacity=0.8}{VerySoftBlue,opacity=0.2};

    \end{tikzpicture}    \caption{A generalised hedgehog with 6 vertices in its body and 9 spikes.}
    \label{fig:hedgehog}
\end{figure}
In this note we settle the question completely by finding a  $1$-degenerate 3-uniform hypergraph with non-linear 2-colour Ramsey number. This counterexample is a variant on the \emph{hedgehog}, defined by Kostochka and R\"odl~\cite{KosRod}. A {hedgehog} is a 3-uniform hypergraph $H$ defined on vertex set $B\sqcup S$, where $B$ is referred to as the \emph{body} of $H$ while $S$ is the set of \emph{spikes}.
The \emph{standard hedgehog} has $\size{S} = \binom{\size{B}}{2}$ and each pair of vertices in the body forms an edge with exactly one spike.  
In a \emph{(generalised) hedgehog}, we only require that each spike vertex form an edge with exactly one pair of vertices in the body (see Figure~\ref{fig:hedgehog}). Notice that any hedgehog is 1-degenerate.

The standard hedgehog (and its higher uniformity generalisations) have been well studied and provide most of the constructions mentioned above. In addition, famously, the standard hedgehog is `not colour-blind': Conlon, Fox and R\"odl~\cite{CFR} showed that its $k$-colour Ramsey number varies from polynomial in the number of vertices when $k=2$ to exponential \mbox{when~$k=4$,} providing the first example where a small increase in the number of colours causes such a jump. Fox and Li~\cite{FoxLi} subsequently showed that the $2$-colour Ramsey number of the standard hedgehog on $n$ vertices is actually almost linear: it is $O(n\log n)$.

\medskip

We use generalised hedgehogs to prove that there exist counterexamples with degeneracy one in the 2-colour setting. This is our main contribution.
\begin{theorem}\label{thm:genheg-lb} 
There exists a constant $c>0$ such that, for all sufficiently large $n$, there exists an $n$-vertex generalised hedgehog $H^*$ with
\[R(H^*) \geq c\frac{n^{3/2}}{\log n}.\]
\end{theorem}

We complement the above result with a near-sharp upper bound for the $2$-colour Ramsey numbers of generalised hedgehogs. In fact, the upper bound holds in the more general asymmetric setting. 
\begin{theorem}\label{thm:genheg-ub} 
There exists a constant $C>0$ such that, for any sufficiently large $n$ and any $n$-vertex generalised hedgehogs $H,H'$, we have 
\[R(H,H')\le Cn^{3/2}\,.\]
\end{theorem}

\section{The proof}

We need the following auxiliary result, which is proved using standard techniques for random graphs. 
\begin{lemma}\label{lemma:randomgraphs}
    For all sufficiently large $n\in \mathbb{N}$, there exists a graph $G$ on ${n^{3/2}}/({10^{6}\log n})$ vertices such that
    \begin{enumerate}[label=\upshape {(\roman{enumi})}]
    \item each vertex $v\in V(G)$ has degree at most $3n/(2\cdot 10^3)$;\label{random:degrees}
    \item $G$ contains no clique on 10 vertices;\label{random:clique}
    \item $G$ has no independent set of size $\sqrt{n}/50$.\label{random:independent}
\end{enumerate}
\end{lemma}
\begin{proof}
    Let $N={n^{3/2}}/({10^{6}\log n})$ and $p=800n^{-1/2}\log n$. Recall that $G(N,p)$ denotes the $N$-vertex binomial random graph where each possible  edge is included independently with probability $p$. Let $G\sim G(N,p)$. Chernoff's inequality says that the probability that a given vertex of $G(N,p)$ has degree more than $\tfrac32pN$ is at most $\exp(-\tfrac1{12}pN)\le N^{-2}$, and so by the union bound, with  probability at least $1-1/n$, every vertex of $G$ has degree at most $3pN/2 \le 3n/(2\cdot 10^3)$. It easily follows from a first moment argument that with high probability~$G$ contains no clique of order ten. Indeed, the probability that a fixed set of ten vertices induces a clique is~$p^{\binom{10}{2}} = p^{45}$; by the union bound, the probability that there exists a copy of $K_{10}$ is bounded by~$\binom{N}{10}p^{45} < 1/n$.  Similarly, setting $a = 10p^{-1}\log N \leq \sqrt{n}/50$, the probability that a fixed set of $a$ vertices is independent is $(1-p)^{\binom{a}{2}}$. The probability that $G$ contains an independent set of size $a$ is then at most
    \[\binom{N}{a}(1-p)^{\binom{a}{2}} < \exp\big(a\log N-\tfrac14pa^2\big)<\exp\big(-a\log N\big) < 1/n\,.\]
    Hence, there exists an instance satisfying all of~\ref{random:degrees},~\ref{random:clique} and~\ref{random:independent}, which we take as our graph~$G$.  
\end{proof}

\newcounter{propcounter}
\stepcounter{propcounter}

Our main contribution is the following construction, proving Theorem~\ref{thm:genheg-lb}.

\begin{proof}[Proof of Theorem~\ref{thm:genheg-lb}] We first define $H^*$ as follows. Let the body $B$ of $H^*$ have size $\sqrt{n}/50$. To every pair in $B$ attach one spike. In addition, to each pair among the first $10$ vertices of $B$, attach a further $n/100$ spikes. So far the graph contains $\tfrac{\sqrt{n}}{50} + \binom{\sqrt{n}/50}{2} + \binom{10}{2}\tfrac{n}{100} \leq n$ vertices.
Add spikes to reach $n$ vertices and, for each new spike, add an arbitrary edge.

Letting $c=10^{-3}$ and  $N=c^2 n^{3/2}/\log n$, Lemma~\ref{lemma:randomgraphs} gives an $N$-vertex graph $\Gamma$ on vertex set~$V$ such that
\begin{enumerate}[label=\textbf{\Alph{propcounter}\arabic{enumi}}]
    \item\label{A:1} each vertex in $V$ has degree at most $3cn/2$; 
    \item\label{A:2} $\Gamma$ contains no clique on $10$ vertices; 
    \item\label{A:3} $\Gamma$ has no independent set of size $\sqrt{n}/50$.
\end{enumerate}
We define a complete $2$-coloured $3$-graph $G$ on $V$ as follows. For a triple $uvw$, if one or more of $uv,uw,vw$ is an edge of $\Gamma$, we colour $uvw$ red, and, otherwise, we colour it blue.
\begin{figure}[h!]
   \centering\begin{tikzpicture}[thick, scale=.7]
    \tikzstyle{every node}=[circle, draw, fill, inner sep=0pt, minimum width=4pt]
    
    \draw node (u1) at (0,0){};
    \draw node (v1) at (1,0){};
    \draw node (w1) at (.5,1){};

    \draw node (u2) at (3,0) {};
    \draw node (v2) at (4,0) {};
    \draw node (w2) at (3.5,1) {};
  
   \draw node (u3) at (6,0) {};
    \draw node (v3) at (7,0) {};
    \draw node (w3) at (6.5,1) {};

    \draw node (u4) at (9,0) {};
    \draw node (v4) at (10,0) {};
    \draw node (w4) at (9.5,1) {};
  

    \draw (u1) -- (v1) -- (w1) -- (u1);

    \draw (u2) -- (v2) -- (w2);
    \draw[dashed] (w2) -- (u2);

    \draw (u3) -- (v3);
    \draw[dashed] (v3) -- (w3) -- (u3);

    \draw[dashed] (u4) -- (v4) -- (w4) -- (u4);

    \triple{(v1)}{(u1)}{(w1)}{6pt}{1pt}{red,opacity=0.8}{red,opacity=0.2};
    \triple{(v2)}{(u2)}{(w2)}{6pt}{1pt}{red,opacity=0.8}{red,opacity=0.2};
    \triple{(v3)}{(u3)}{(w3)}{6pt}{1pt}{red,opacity=0.8}{red,opacity=0.2};
    \triple{(v4)}{(u4)}{(w4)}{6pt}{1pt}{blue,opacity=0.8}{blue,opacity=0.2};   
   \end{tikzpicture}
   \caption{The possible configurations and corresponding coloured edge in the 3-graph $G$, where a dashed line represents a non-edge in the 2-graph $\Gamma$.}
\end{figure}
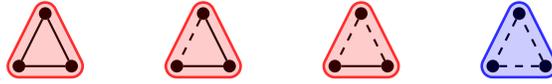

Suppose now that $G$ contains a blue copy of $H^*$, and let $B'$ be the image of the body of~$H^*$ in $V$. As each pair of $B'$ must be in at least one blue edge of $G$ in the copy of $H^*$, $B'$ is an independent set in $\Gamma$ of size $\sqrt{n}/50$, contradicting~\ref{A:3}.

Suppose instead that $G$ contains a red copy of $H^*$. Observe that, if $uv$ is a non-edge in~$\Gamma$, then any red edge $uvw$ of $H^*$ necessarily contains an edge $uw$ or $vw$ of $\Gamma$. In particular, $uv$ is in at most $d_\Gamma(u)+d_\Gamma(v)\le 3cn$ red edges of $G$ (because of~\ref{A:1}). Consider the set $K$ which is the image of the first $10$ vertices of the body of $H^*$. As each pair in $K$ is in at least $n/100>3cn$ red edges of $G$, $\Gamma[K]$ is a clique on $10$ vertices, contradicting~\ref{A:2}.
\end{proof}

We note that this construction also shows $R(S_n,H^*)=\Omega(n^{3/2}/\log n)$, where $S_n$ is the standard hedgehog with $n$ vertices.
\bigskip

We now turn our attention to the upper bound.
Its proof is a small modification of arguments of Conlon, Fox and R\"odl~\cite{CFR}, beginning with the following lemma which we extract from their work.

\begin{lemma}\label{lem:onlyone}
 Let $H$ be a $2$-coloured complete $3$-graph, and let $d_r,d_b$ be integers such that $H$ has at least $d_r+d_b+1$ vertices. Fix $u\in V(H)$, and colour the pairs $uv$ red if $uv$ is in fewer than~$d_r$ red triples, and blue if in fewer than $d_b$ blue triples. Then $u$ is in fewer than $2d_b$ red pairs, or in fewer than $2d_r$ blue pairs.
\end{lemma}
\begin{proof}
  Suppose for a contradiction that $u$ forms a red pair with a set $R$ of size $2d_b$ and a blue pair with a set $B$ of size $2d_r$. Observe that a pair $uv$ cannot be both blue and red, since it is in at least $d_r+d_b-1$ edges of $H$; hence, the sets $R$ and $B$ are disjoint.
 
 The number of red edges containing $u$ and a vertex of $R$ is by definition less than $d_r|R|=2d_rd_b$; similarly, the number of blue edges containing $u$ and a vertex of $B$ is less than $d_b|B|=2d_rd_b$. However, the total number of edges  $uxy$ with  $x\in R$ and $y\in B$ is $|R||B|=4d_rd_b$, a contradiction.
\end{proof}

With the lemma in hand, we can complete the proof of the main theorem; again, this follows the ideas of~\cite{CFR} but requires a little more care.

\begin{proof}[Proof of Theorem~\ref{thm:genheg-ub}]
 Let $N=10n^{3/2}$, and let $G$ be an $N$-vertex $2$-coloured complete $3$-graph. We define an auxiliary graph $\Gamma$ on $V(G)$, where we connect two vertices by a red edge if they lie in fewer than $n$ red triples in $G$, and blue if in fewer than $n$ blue triples. We now mark a vertex as red if it is in at most $2n$ red edges in $\Gamma$, and blue otherwise. By Lemma~\ref{lem:onlyone}, every blue vertex in $\Gamma$ is in at most $2n$ blue edges. By the pigeonhole principle,  there are at least~$N/2$ red vertices or at least $N/2$ blue vertices. Suppose without loss of generality that there are at least~$N/2$ red vertices and call the set of red vertices $V_1$; we will embed a red hedgehog in $G$ whose body will lie in~$V_1$.
 
 Given an $n$-vertex generalised hedgehog $H$, we define a graph $F$ on the body of $H$, where the edges of $F$ are the pairs that form an edge with some spike of $H$. Observe that $F$ has at most $n$ edges, and suppose $F$ has degeneracy exactly $D$. Then, there is a subgraph $F'$ of~$F$ such that every vertex of $F'$ has degree at least $D$. It follows that $F'$ has at least $D+1$ vertices and hence at least $\binom{D+1}{2}$ edges. Since $e(F)\leq n$, we have $D\le2\sqrt{n}$.
 
 Fix a $D$-degeneracy order on the vertices of $F$. We begin to embed $H$ into $G$ by embedding the vertices of the body in the degeneracy order, at each embedding of a vertex $x\in V(H)$ choosing an image $u\in V_1$ so that the following property is maintained: the edges of $F$ are all mapped to pairs that do not form red edges in $\Gamma$. By definition, adding $x$ embeds at most $D$ edges of $F$, whose other endpoints are mapped to some vertices $v_1,\dots,v_d\in V_1$, where $d\le D$. Each $v_i$ has red-degree at most $2n$ in $\Gamma$. It follows that, since $|V_1|\geq 5n^{3/2} > 2Dn+n$, there will exist a vertex of $V_1$ which was not previously used in the embedding and which does not form a red edge with any of $v_1,\dots,v_d$.
 
 We now greedily extend this embedding to an embedding of $H$. For each spike vertex of~$H$, pick greedily a vertex which is not yet used and forms a red edge with its corresponding pair $e$ in the body. Since $e$ is an edge of $F$, it is not mapped to a red edge of $\Gamma$, so there are at least~$n$ red $3$-edges of $G$ using $e$; thus not all candidate vertices have been used in the embedding so far.
\end{proof}

\section{Concluding remarks}

It would be interesting to close the $\log$-factor gap in Theorems~\ref{thm:genheg-lb} and~\ref{thm:genheg-ub}, even specifically for the type of graph $H^*$ giving the lower bound. One argument in favour of the lower bound is that our upper bound proof for $H^*$ uses the fact that we can find an independent set of size $\tfrac{N}{\Delta(R)+1}$ in the graph $R$ of red edges. This would of course be improved if we knew that~$R$ contained no or few copies of $K_{10}$. On the other hand, if $R$ does contain a copy of $K_{10}$, then this provides us with the high-degree part of a blue copy of $H^*$, which one might hope to extend to a blue $H^*$, given Fox and Li's result~\cite{FoxLi} that the Ramsey number of the standard hedgehog is $O(n\log n)$.

\medskip
We would like better upper bounds on the $2$-colour Ramsey numbers of general $1$-degenerate hypergraphs. By iteratively removing a maximal set of degree-one vertices, we see that any $1$-degenerate hypergraph $H$ can be recursively decomposed into generalised hedgehogs.

\begin{lemma}\label{lemma:hedgehogs}If $H$ is a $1$-degenerate $3$-graph with no isolated vertices, then there are edge-disjoint subgraphs $H_1,\ldots, H_t\subset H$ such that $H_i$ is a $3$-uniform generalised hedgehog for each $i\in [t]$, and $H=\bigcup_{i\in [t]}H_i$.
\end{lemma}

Suppose we write $H$ as the union of $t$ generalised hedgehogs as in Lemma~\ref{lemma:hedgehogs}. Iterating an argument similar to that of Theorem~\ref{thm:genheg-ub}, we can show that $R(H)=n^{O(t)}$. For small $t$ this is a relatively good bound, but in general $t$ can be linear in $n$, which gives an embarrassingly weak  general upper bound of $R(H)=n^{O(n)}$; nevertheless, we do not know of anything better.

\section*{Acknowledgments}
This project was initiated at the UCL Workshop in Extremal and Probabilistic Combinatorics in 2023. We are grateful to the organisers, Shoham Letzter, Kyriakos Katsamaktsis, and Amedeo Sgueglia, for inviting us and to Alexey Pokrovskiy for proposing this problem.
\smallskip

S.B.: Most of this research was conducted while the author was at the School of Mathematics, University of Birmingham, Birmingham, United Kingdom. The research leading to these results was supported by EPSRC, grant no.\ EP/V048287/1 and by ERC Advanced Grants ``GeoScape'', no.\ 882971 and ``ERMiD'', no.\ 101054936. There are no additional data beyond that contained within the main manuscript. 
\smallskip

M.P.S.: Most of this research was conducted while the author was at the Mathematics Institute, University of Warwick, supported by the European
Research Council (ERC) under the European Union Horizon 2020 research and innovation programme (grant agreement No.\ 947978). Also supported by ANID Fondecyt Regular grant No. 1241398 and ANID Basal Grant CMM FB210005.

\bibliographystyle{amsplain}
\bibliography{bib}

\end{document}